\documentclass{amsart}
 
\newtheorem{theorem}{Theorem}[section]

\newtheorem{proposition}[theorem]{Proposition}
\newtheorem{corollary}[theorem]{Corollary}

\theoremstyle{definition}
\newtheorem{definition}[theorem]{Definition}
\newtheorem{example}[theorem]{Example}

\theoremstyle{remark}

\numberwithin{equation}{section}

\def\t{\mbox{tr}\,} 

\newcommand\style{\mathcal }          

\newcommand{\B}{\style{B}}
\newcommand{\M}{\style{M}}

\renewcommand{\H}{\style{H}}
\newcommand{\K}{\style K}

\newcommand\cstar{{\rm C}^*}                              
 
\newcommand\spec{{\rm Sp}}      

\begin{document}

\title[Matrix Convexity and Dilations of Toeplitz-Contractive Tuples]{Matrix Convexity and 
Unitary Power Dilations of Toeplitz-Contractive Operator Tuples}

\author{Douglas Farenick}
\address{Department of Mathematics and Statistics, University of Regina, Regina, Saskatchewan S4S 0A2, Canada}
\curraddr{}
\email{douglas.farenick@uregina.ca}

\begin{abstract} Using works of T.~Ando and L.~Gurvits, the
well-known theorem of P.R.~Halmos concerning the existence of unitary dilations for 
contractive linear operators acting on Hilbert spaces recast as a result 
for $d$-tuples of contractive Hilbert space operators satisfying a certain matrix-positivity condition. 
Such operator $d$-tuples satisfying this matrix-positivity condition are called, herein, 
Toeplitz-contractive, and a characterisation of the Toeplitz-contractivity condition is presented. 
The matrix-positivity condition leads to definitions of new distance-measures in several variable operator theory, 
generalising the notions of norm, numerical radius, and spectral radius to $d$-tuples of operators (commuting, for the spectral radius)
in what appears to be a novel, asymmetric way. 
Toeplitz contractive operators form a noncommutative convex set, and a scaling constant $c_d$ for inclusions of
the minimal and maximal matrix convex
sets determined by a stretching of the unit circle $S^1$ across $d$ complex dimensions is shown to exist.
\end{abstract}

\maketitle

 
 \section{Introduction}
 
The influential, foundational paper of Paul Halmos \cite{halmos1950} in 1950 examined the 
notions of dilation, compression, and 
extension in the algebra $\B(\H)$ of bounded linear operators acting on a complex Hilbert space $\H$.  
If $T:\H\rightarrow\H$ is a bounded linear operator, then a \emph{dilation} of $T$ is a bounded linear operator
 $S:\K\rightarrow\K$ acting on a Hilbert space $\K$ that contains $\H$ as a subspace
such that, for every $\xi\in H$, 
 \begin{equation}\label{e:dil}
 T\xi=P_{\H} S\xi,
 \end{equation} 
 where $P_{\H}$ is the linear selfadjoint projection operator on $\K$ whose range is 
 $\H$ and kernel is $\H^\bot$. Conversely, given a bounded linear operator $S$ 
 acting on $\K$ and a subspace $\H$ of $\K$, equation \eqref{e:dil}
 defines a bounded linear operator $T$ on $\H$, which is called the \emph{compression} of $S$. 
 A dilation $S$ of $T$ is an \emph{extension} of $T$ if the subspace $\H$ of $\K$ is invariant under the action of $S$.

 The general goal of dilation theory is to study an operator $T$, which may be difficult to analyse directly, by dilating $T$
 to an operator $S$ that
 enjoys well understood properties, and then using $S$ to analyse $T$. 
 For example, if $S$ is a dilation of $T$, then equation (\ref{e:dil}) indicates $\|T\|\leq \|S\|$; thus,
 if $S$ is a unitary operator, then every compression $T$ of $S$ is a contraction (i.e., $\|T\|\leq1$).
The theorem below of Halmos \cite{halmos1950} shows the converse is true. 
  
 \begin{theorem}[Halmos]\label{h1} 
 If $T$ is a contraction, then $T$ has a unitary dilation.
 \end{theorem}
 
 Halmos' proof of Theorem \ref{h1} is constructive in that 
 he shows the dilating Hilbert space can be taken to be $\K=\H\oplus\H$ and gives an explicit description of the unitary dilation
 $U$ as a $2\times 2$ matrix of operators. 
 
 Because $2\times 2$ matrices of operators can be used to recover the norm of a single operator via the formula
  \begin{equation}\label{e:n1}
 \|T\| = \inf\left\{\alpha\in\mathbb R_+\,|\,\left[\begin{array}{cc} \alpha 1_\H & T^* \\ T & \alpha 1_\H \end{array}\right]\mbox{ is a positive operator on }
 \H\oplus\H\right\},
 \end{equation}
 where $1_\H$ denotes the identity operator on $\H$, an equivalent formulation of Halmos' Theorem 
 can be made in terms of matrix
 positivity.
 
 \begin{theorem}[Halmos, Alternative Version]\label{h2} The following statements are equivalent for a bounded linear operator $T$ acting on a Hilbert space $\H$.
 \begin{enumerate}
 \item the operator matrix $\left[ \begin{array}{cc} 1_\H & T^* \\ T & 1_\H \end{array}\right]$, acting on $\K=\H\oplus\K$, 
 is a positive operator on $\K$;
 \item there exists a unitary operator $U\in\B(\K)$ such that $U$ is a dilation of $T$.
 \end{enumerate}
 \end{theorem}
 
The purpose of this paper is to bring together a variety of known results 
to show that a ``matrix positivity'' property leads to a version of Halmos's Theorem \ref{h2}
for not just one contraction, but for $d$ contractive operators, leading to new distance measures involving
operator $d$-tuples and new forms of joint numerical and spectral radii that mimic properties well-known to hold
for single operators.

 \begin{definition}[Toeplitz-contractive $d$-tuples]\label{tc def} A $d$-tuple 
 ${\sf T}=(T_1,\dots T_d)$ of bounded linear operators
 on a Hilbert space $\H$ is \emph{Toeplitz-contractive} if the Toeplitz operator-matrix
 \[
 \mathcal T= \left[ \begin{array}{ccccc} 1_\H & T_1^* & T_2^* & \dots & T_d^* \\
                                                              T_1 & 1_\H & T_1^* & \ddots & \vdots \\
                                                              T_2 & \ddots & \ddots & \ddots &  T_2^* \\
                                                              \vdots & \ddots & \ddots & \ddots & T_1^* \\
                                                              T_d & \dots & T_2 & T_1 & 1_\H 
 \end{array}\right] 
 \]
 is a positive operator on $\K=\displaystyle\bigoplus_1^{d+1}\H$.
 \end{definition}

Definition \ref{tc def} is not the usual way one defines contractivity of an operator $d$-tuple.  
One of the most widely used definitions is that of a row contraction.

\begin{definition}[Row-contractive $d$-tuples]\label{rc def}
A $d$-tuple 
 ${\sf T}=(T_1,\dots T_d)$ of operators $T_k$ on $\H$ is \emph{row contractive} if the 
 linear map $\displaystyle\bigoplus_1^d\H \rightarrow \H$ given by
 \[
 (\xi_1,\dots\,\xi_d) \longmapsto \sum_{k=1}^d T_k\xi_k
 \]
 is a contraction.
 \end{definition}
 
 As with single operators, the row-contractivity condition has an equivalent formulation as 
 a positivity condition:
 a $d$-tuple ${\sf T}$ is row-contractive 
 if and only if $1_\H-\displaystyle\sum_{k=1}^d T_kT_k^*$ is a positive operator.
 
 The simplest example of a Toeplitz-contractive $d$-tuple is the following one.
 
 \begin{example}\label{ex:power contraction}
 If $T\in\B(\H)$ is a contraction, then 
 ${\sf T}=(T, T^2, \dots, T^d)$ is Toeplitz-contractive.
 \end{example}
 
 \begin{proof} See \cite[pp. 13-14]{Paulsen-book} for a direction calculation showing
 ${\sf T}$ is Toeplitz contractive. 
\end{proof}
  
 \begin{corollary}\label{ex:prototype} If 
$\phi:\B(\H)\rightarrow\B(\K)$ is a unital completely positive
linear map, and if $U\in\B(\H)$
is unitary, then 
\[
\phi({\sf U})=\left(\phi(U),\phi(U^2),\dots,\phi(U^d)\right)
\]
is Toeplitz contractive, for every $d\in\mathbb N$.
\end{corollary}

\begin{proof} Because $\phi$ is completely positive, applying $\phi$ to each entry of the positive operator-matrix
\[
\mathcal U =  \left[ \begin{array}{ccccc} 1_\H & U^* & (U^2)^* & \dots & (U^d)^* \\
                                                             U & 1_\H & U ^* & \ddots & \vdots \\
                                                              U^2 & \ddots & \ddots & \ddots &  (U^2)^* \\
                                                              \vdots & \ddots & \ddots & \ddots & U ^* \\
                                                              U^d & \dots &U^2 & U  & 1_\H 
 \end{array}\right] 
\]
yields a new positive Topeplitz matrix of operators, which is precisely the one induced by the 
$d$-tuple $\phi({\sf U})=\left(\phi(U),\phi(U^2),\dots,\phi(U^d)\right)$.
\end{proof}

 Corollary \ref{ex:prototype} demonstrates that, for a given unitary $U\in\B(\H)$ and subspace $\mathcal L$ of $\H$, 
the compressions of $U^k$
to $\mathcal L$ lead to a sequence of operators 
$T_k=P_{\mathcal L}U^k_{\vert\mathcal L}$ on $\mathcal L$ such that the $d$-tuple $(T_1,\dots,T_d)$ is Toeplitz contractive, 
for every $d\in\mathbb N$. 
The main result of this paper, Theorem \ref{main result}, establishes the converse: 
that all Toeplitz-contractive $d$-tuples result through compressing powers of unitaries to a common subspace.

The study of norms and dilations of $d$-tuples of operators has a long history. In some works, such 
as \cite{binding--farenick--li1995},
normal dilations are invoked to define a norm on operator $d$-tuples. This is also done in the present paper, resulting in a metric rather than
a norm. Although a substantial literature has been developed concerning models for row-contractive $d$-tuples, dating back to \cite{bunce1984}, 
modern approaches explore, as in 
\cite{davisdon--dor-on--shalit--solel2017,passer--shalit--solel2018},
the
close connection between normal dilations of $d$-tuples of operators and matrix convex sets, yielding applications to the matrix ranges
\cite{smith--ward1980} of operators.
\section{Further Examples, Notation, and Definitions}

The Cartesian product of $d$-copies of $\B(\H)$ is denoted by $\B(\H)^d$.
 
If $\phi:\B(\H)\rightarrow\B(\K)$ is a linear map and if
$p\in\mathbb N$, then the $p$-th ampliation of $\phi$ is denoted by $\phi^{[p]}$
and designates the linear map $\M_p\left(\B(\H)\right)\rightarrow \M_p\left(\B(\K) \right)$ in which
\[
 \phi^{[p]}\left(\left[ X_{ij}\right]_{i,j=1}^p\right)=\left[ \phi(X_{ij}) \right]_{i,j=1}^p,
\]
for every $p\times p$
operator matrix $\left[ X_{ij}\right]_{i,j=1}^p\in\M_p\left(\B(\H)\right)$.  
The map $\phi:\B(\H)\rightarrow\B(\K)$ is unital and completely positive (ucp) 
if $\phi(1_\H)=1_\K$ and $\phi^{[p]}$ is a positive linear map for every $p\in\mathbb N$.

Row-contractivity and Toeplitz-contractivity are defined in quite distinctive ways. 
For example, row-contractivity does not depend on the order of the operators in the $d$-tuple, 
whereas Toeplitz-contractivity does, as $T_1$ appears as an entry in the Toeplitz matrix far more
frequently then $T_d$ does. 
The example below confirms these are indeed different concepts.

\begin{example} The notions of row-contractive and Toeplitz-contractive are distinct in that 
neither one of them implies the other.
\end{example} 

\begin{proof}
If $U$ is a unitary operator,
then Example \ref{ex:power contraction} shows $(U, U^2)$ is Toeplitz-contractive. However, $UU^*+U^2(U^2)^*=2\cdot 1_\H$,
which is not bounded above by $1_\H$. Hence, $(U, U^2)$ is Toeplitz-contractive, but not row-contractive. 

Conversely, if $P$ and $Q$ are commuting nonzero
projections such that $PQ=0$, then $(P,Q)$ is row contractive. Consider
the Toeplitz matrix
\[
\mathcal P =\left[ \begin{array}{ccc} 1_\H & P & Q \\ P & 1_\H & P \\ Q & P & 1_\H \end{array}\right].
\]
Select a unit vector $\xi\in\H$ in the range of $P$, and consider the 
state $\phi$ given by $\phi(X)=\langle X\xi,\xi\rangle$, for bounded linear operators $X$ on $\H$. 
Thus, $\phi(1_\H)=\phi(P)=1$
and $\phi(Q)=0$; therefore, applying $\phi$ to each entry of $\mathcal P$ yields the $3\times 3$ complex matrix 
\[
\phi^{[3]}(\mathcal P)= \left[ \begin{array}{ccc} 1 & 1 & 0 \\ 1&1&1 \\ 0&1&1 \end{array}\right].
\]
However, the matrix $\phi^{[3]}(\mathcal P)$ is not a positive operator 
on $\mathbb C^3$, which implies the matrix $\mathcal P$ cannot be a positive operator on $\H$
(because $\phi^{[3]}$ is a positive linear map). 
Hence, the row contractive 
pair $(P,Q)$ is not Toeplitz contractive. 
\end{proof}

By a similar argument, we note:

\begin{example} If $V_1$ and $V_2$ are isometries such that $V_1V_1^*+V_2V_2^*=1_\H$, then 
$(V_1,V_2)$ is row contractive but not Toeplitz contractive.
\end{example}

\begin{example} \label{pairs}
The following statements are equivalent for $T\in\B(\H)$:
\begin{enumerate}
\item $(T,0)$ is Toeplitz contractive;
\item $(T,T^*)$ is row contractive.
\end{enumerate}
\end{example}

\begin{proof} The matrix 
\[
\left[ \begin{array}{ccc} 1_\H & T^* & 0 \\ T &1_\H &T^* \\ 0 & T & 1_\H \end{array}\right]
\]
is positive if and only if $1_H-TT^*-T^*T$
is positive \cite[Lemma 6.3]{farenick--kavruk--paulsen2013}.
\end{proof}

The operator-theoretic notions of normality, subnormality, and unitary take the following forms in several variables.

\begin{definition}\label{d:nusn} A $d$-tuple ${\sf T}=(T_1,\dots T_d)\in\B(\H)^d$ is:
\begin{enumerate}
\item \emph{normal}, if $T_1$,\dots, $T_d$ are commuting normal operators;
\item \emph{unitary}, if $T_1$,\dots, $T_d$ are commuting unitary operators;
\item \emph{power unitary}, if there is a unitary operator $U\in\B(\H)$ such that $T_k=U^k$, for each $k=1,\dots, d$;
\item \emph{subnormal}, if $T_1$,\dots, $T_d$ are pairwise
commuting operators such that there exist a Hilbert space $\K$ containing $\H$ as a 
subspace and commuting normal operators $N_1,\dots,N_d\in\B(\K)$ such that, for every $k=1,\dots, d$, 
$\H$ is an invariant subspace for $N_k$ 
and the restriction of
$T_N$ to $\H$ is given by $T_k$.
\end{enumerate}
\end{definition}

It is known that a $d$-tuple of commuting subnormal operators may fail to be subnormal in the sense of (4) above; that is,
there exist $d$ commuting subnormal operators that fail to have an extension to $d$ commuting normal operators.

Of the various definitions 
of joint spectrum for commuting operators, the following two are required. Recall from \cite{taylor1970} that every $d$-tuple
${\sf T}\in\B(\H)^d$ of pairwise commuting operators and $\lambda\in\mathbb C^d$
determine a chain complex, denoted by $E({\sf T}-\lambda,\H)$, called the Koszul complex.

\begin{definition}\label{d:js}
The \emph{Gelfand joint spectrum} of a normal $d$-tuple ${\sf N}=(N_1,\dots,N_d)$
is the subset
$\spec_G({\sf N})\subset\mathbb C^d$ 
of all $\lambda\in\mathbb C^d$ for which there exists a unital homomorphism $\pi:\cstar({\sf N})\rightarrow\mathbb C$
such that $\lambda_j=\pi(N_j)$ for all
$j=1,\dots,d$, where $\cstar({\sf N})$ is the unital abelian C$^*$-algebra generated by $N_1,\dots,N_d$.
\end{definition}

\begin{definition}\label{d:ts}
The \emph{Taylor joint spectrum} of a $d$-tuple ${\sf T}=(T_1,\dots,T_d)$
of pairwise commuting (possibly nonnormal) operators
is the subset
$\spec_T({\sf T})\subset\mathbb C^d$ 
of all $\lambda\in\mathbb C^d$ 
for which the Koszul complex $E({\sf T}-\lambda,\H)$ is not exact.
\end{definition}

Lastly, the following forms of (joint) numerical range and matrix range 
\cite{davisdon--dor-on--shalit--solel2017,smith--ward1980}
are considered. 

\begin{definition}\label{d:qaz} Assume that ${\sf T}=(T_1,\dots,T_d)\in\B(\H)^d$.
The \emph{spatial numerical range} of ${\sf T}$ is the set
\begin{equation}\label{e:jnr}
W({\sf T})=\left\{ \left( \langle T_1\xi,\xi\rangle, \dots, \langle T_1\xi,\xi\rangle\right)\,|\,\xi\in\H, \, \|\xi\|=1\right\};
\end{equation}
the \emph{numerical range} of ${\sf T}$ is the set
\begin{equation}\label{e:jnra}
W_1({\sf T})=\left\{ \left( \phi(T_1), \dots, \phi(T_d)\right)\,|\,\phi\mbox{ is a state on }\B(\H) \right\};
\end{equation}
and the 
\emph{$n$-th matix range} of ${\sf T}$ is the set
\begin{equation}\label{e:jmra}
W_n({\sf T})=\left\{ \left( \phi(T_1), \dots, \phi(T_d)\right)\,|\,\phi\mbox{ is a ucp map }\B(\H)\rightarrow\M_n(\mathbb C) \right\}.
\end{equation}
\end{definition}

The numerical ranges defined in (\ref{e:jnr}) above have been extensively studied in operator theory, especially in the cases where $\H$
has finite dimension. 
In operator and matrix theory, $W({\sf T})$ is usually called the \emph{joint numerical range} of ${\sf T}$.

\begin{definition}\label{d:uTm} 
The \emph{universal positive $n\times n$ Toeplitz matrix} 
is the Toeplitz-matrix-valued function $T_n:S^1\rightarrow\M_n(\mathbb C)$ given by
\begin{equation}\label{e:univ T}
T_n(z)=\left[ \begin{array}{cccccc} 
1 & z^{-1} & z^{-2} &  \dots&  z^{-n+2}& z^{-n+1} \\
z & 1 & z^{-1} & z^{-2}& \dots & z^{-n+2} \\
z^2 & z & 1 & z^{-1} &\ddots&   \vdots\\
\vdots & \ddots & \ddots & \ddots &\ddots &z^{-2} \\
z^{n-2} &   &  \ddots & \ddots &\ddots &z^{-1} \\
z^{n-1} & z^{n-2} & \dots & z^{2} &z&1
\end{array}
\right].
\end{equation}
\end{definition}

The matrix $T_n(z)$ defined in \eqref{e:univ T} admits a factorisation of the form $T_n(z)=\gamma_n(z)\gamma_n(z)^*$, 
where 
\[
\gamma_n(z)=\left[\begin{array}{c}1 \\ z  \\ z^2 \\ \vdots \\ z^{n-1} \end{array}\right].
\]
By functional calculus, if $U$ is any unitary, then the operator matrix $T_n(U)$ is positive, and the first column 
below the $(1,1)$-entry is a power-unitary operator $(n-1)$-tuple.

 \begin{theorem}\label{Toepl ext pts}{\rm (\cite[Proposition 4.8(1)]{connes-vansuijlekom2021})}  
The extremal rays of the cone of positive $n\times n$ Toeplitz matrices
are precisely the positive scalar multiples of Toeplitz matrices of the form $T_n(\lambda)$, for $\lambda\in S^1$.
\end{theorem}

To complete this section, let $S^1$ denote the unit circle in the complex plane and denote
the 
multiplication operator $f\mapsto \varphi\cdot f$
on $L^2(S^1)$, where $\varphi:S^1\rightarrow\mathbb C$ is a continuous function, by $M_\varphi$. 
For each $k\in\mathbb Z$, let $\chi_k(z)=z^k$, for $z\in S^1$; thus, $\chi_{-k}=\overline{\chi_k}$ and $\chi_k=\chi^k$, for every
$k\in\mathbb Z$. By the Stone-Weierstrass Theorem, the unital abelian C$^*$-algebra $C(S^1)$ of complex continous
functions $S^1\rightarrow \mathbb C$ coincides with the unital 
C$^*$-algebra $\cstar(\chi)$ generated by $\chi$, which in turn is the group C$^*$-algebra generated by $\mathbb Z$. Thus,
$\chi$ is a universal unitary, which is to say that if $U$ is any unitary operator acting on $\B(\H)$, then there is a unital
$*$-homomorphism $\pi:C(S^1)\rightarrow\B(\H)$ such that $\pi(\chi)^k=U^k$, for every $k\in\mathbb Z$. Because the 
unital $*$-homomorphism $\pi:C(S^1)\rightarrow \B\left(L^2(S^1)\right)$ in which $\pi(\varphi)=M_\varphi$, for every $\varphi\in C(S^1)$,
is an isometry, the unitary operator $M_\chi$ is also a universal unitary. In representing $M_\chi$ as a doubly infinite (Laurent)
matrix relative to the canonical orthonormal basis of $L^2(S^1)$, the operator
$M\chi$ is unitarily equivalent to the bilateral shift operator
$W$ acting on the Hilbert space $\ell^2(\mathbb Z)$. Thus, the bilateral shift operator is a universal unitary operator.

 \section{The Dilation Theorem}
 
The following theorem connects various results in the literature to
characterise Toeplitz-contractive $d$-tuples in terms of norm conditions and unitary power dilations. The key step, asserting
the existence of the dilation, is achieved in three separate ways: 
by applying Ando's theorem \cite{ando1970,ando2013} on truncated moment problems;
by appealing to Gurvits' theorem \cite{gurvits2001,gurvits--burnam2002}
on the separability of positive block-Toeplitz matrices; or by invoking the theory of 
pure completely positive linear maps \cite{farenick--mcburney2023}.

 \begin{theorem}\label{main result} The following statements are equivalent for a  
 $d$-tuple ${\sf T}=(T_1,\dots T_d)$ of bounded linear operators acting 
 on a Hilbert space $\H$:
 \begin{enumerate} 
 \item[{(a)}] the $d$-tuple ${\sf T}$ is Toeplitz-contractive;
 \item[{(b)}] for all $m\times m$ complex matrices $A_0,A_1,\dots A_d$, and all positive integers $m$,
 \[
 \left\| 1_\H \otimes A_0 + \sum_{k=1}^d T_k\otimes A_k\right\| \leq \max_{z\in\mathbb C,\,|z|=1}\left\| A_0 + \sum_{k=1}^d z^kA_k\right\|; 
 \]
 \item[{(c)}] there exists a Hilbert space $\K$ containing $\H$ as a subspace and a unitary operator $U$ on $\K$
 such that $U^k$ is a dilation of $T_k$, for each $k=1,\dots,d$.
 \end{enumerate}
 If the Hilbert space $\H$ is separable, then there is an additional equivalent assertion:
 \begin{enumerate} 
 \item[{(d)}] for every $\varepsilon>0$, there is an isometry $V_\varepsilon:\H\rightarrow\ell^2(\mathbb Z)$ such that
 \[
 \|T_k-V_\varepsilon^* W^k V_\varepsilon\|<\varepsilon, 
 \]
 for each $k=1,\dots,d$, where $W$ denotes the bilateral shift operator on $\ell^2(\mathbb Z)$.
 \end{enumerate}
 Moreover, if $\H$ has finite dimension and the equivalent 
 conditions (a)-(d) hold, then there exists a dilating Hilbert space $\K$, in (c), 
 such that $\K$ has finite dimension.
 \end{theorem}
 
 \begin{proof} As mentioned above,
 there are three distinct proofs of the implication $\mbox{(a)}\Rightarrow\mbox{(c)}$. 
 
 The first proof is achieved using 
 Ando's Theorem 
 \cite{ando1970,ando2013} on truncated moment problems, which
shows that each $T_k$ is the image of $z^k\in C(S^1)$ under a ucp map $\phi:C(S^1)\rightarrow\B(\H)$. By expressing
$\phi$ in terms of its Stinespring decomposition $\phi=V^*\pi V$, for some representation $\pi:C(S^1)\rightarrow\B(\H_\pi)$
and linear isometry $V:\H\rightarrow\H_\pi$, then
the dilation is achieved by taking $\K=\H_\pi$, $U=\pi(z)$, $P=VV^*$, and considering $T$ as the operator on the subspace
$V(\H)\subset\K$ given
by $\eta\mapsto VTV^*\eta$, for $\eta\in V(\H)$.

The
second proof is via Gurvits' proof of his theorem \cite{gurvits2001,gurvits--burnam2002} on separable positive
block-Toeplitz matrices, 
under the hypothesis that $\H$ has finite dimension; in this case,  Gurvits' methods show the resulting dilating space $\K$
can be taken to have finite dimension.

The third proof is due to the
author and McBurney \cite{farenick--mcburney2023}, based on the theory of pure completely positive linear maps. The
proof in \cite{farenick--mcburney2023} requires no restriction on the dimension of $\H$; 
however, if $\H$ has finite dimension, then it is shown in 
 \cite{farenick--mcburney2023} that the dilating space $\K$ can also be taken to be finite-dimensional.
 
 To prove the implication $\mbox{(c)}\Rightarrow\mbox{(b)}$, let $(U,\dots,U^d)$ be a unitary dilation of $(T_1,\dots,T_d)$.
 Thus, $T_k\xi = P_\H U^k\xi$, for every $\xi\in H$ and $k=1,\dots,d$. If $m$ is a positive integer and $A_0,A_1,\dots A_d$
 are $m\times m$ complex matrices, then the operators
 \[
 1_\H \otimes A_0 + \sum_{k=1}^d T_k\otimes A_k \,\mbox{ and }\,
 1_\K \otimes A_0 + \sum_{k=1}^d U^k\otimes A_k
 \]
act on the Hilbert spaces $\H\otimes\mathbb C^m$ and $\K\otimes\mathbb C^m$, respectively. 
Moreover, 
\[
1_\H \otimes A_0 + \sum_{k=1}^d T_k\otimes A_k \,\mbox{ is a compression of }\,
1_\K \otimes A_0 + \sum_{k=1}^d U^k\otimes A_k
\]
to the subspace $\H\otimes\mathbb C^m$ of $\K\otimes\mathbb C^m$. Hence, 
\[
\left\| 1_\H \otimes A_0 + \sum_{k=1}^d T_k\otimes A_k \right\| \leq 
\left\| 1_\K \otimes A_0 + \sum_{k=1}^d U^k\otimes A_k \right\|.
\]
Because ${\sf U}=(U,\dots, U^d)$ is a $d$-tuple of
 pairwise commuting unitary operators, the maximal ideal space $\mathfrak M$ of the abelian C$^*$-algebra 
$\cstar({\sf U})$ is homeomorphic to $\mbox{Sp}_G({\sf U})$ under the map 
$\rho\mapsto (\rho(U),\dots,\rho(U^d))$. Hence, every operator $U^k$
is a complex-valued continuous function on $\mathfrak M$ whereby $\rho\mapsto\rho(U^k)$. 
Because the unital C$^*$-algebra $C(\mathfrak M)\otimes \M_m(\mathbb C)$ is isometrically isomorphic 
to the C$^*$-algebra of continuous functions on $\mathfrak M$ with values in $\M_m(\mathbb C)$, 
we deduce that
\[
\left\| 1_\K \otimes A_0 + \sum_{k=1}^d U^k\otimes A_k \right\|
=
\max_{\rho\in\mathfrak M} \left\| A_0+\sum_{k=1}^d \rho(U^k)  A_k \right\|.
\]
Furthermore, because $\rho(U^k)=\rho(U)^k$ and
\[
\mbox{Sp}_G({\sf U})=\left\{(\lambda,\lambda^2,\dots,\lambda^d)\,|\,\lambda\in \mbox{Sp}_G(U)=\mbox{Sp}(U)\right\}
\subseteq
\left\{(z,z^2,\dots,z^d)\,|\,z \in S^1\right\},
\]
we obtain 
 \[
 \left\| 1_\H \otimes A_0 + \sum_{k=1}^d T_k\otimes A_k\right\| \leq \max_{z\in\mathbb C,\,|z|=1}\left\| A_0 + \sum_{k=1}^d z^kA_k\right\|,
 \]
 thereby completing the proof of the implication $\mbox{(c)}\Rightarrow\mbox{(b)}$.
 
To prove the implication $\mbox{(b)}\Rightarrow\mbox{(a)}$, assume that $T_1,\dots,T_d\in\B(\H)$ are such that
 inequality (b) holds for all matrices $A_k\in\M_m(\mathbb C)$ and positive integers $m$. Let $\mathcal Z_d$
 be the finite-dimensional subspace of $C(S^1)$ spanned by the continuous functions $\chi_k(z)=z^k$, for $k=0,1,\dots, d$, for $z\in S^1$. 
 By the linear independence of $\chi_0,\dots,\chi_d$, there is a unital linear map $\psi:\mathcal Z_d\rightarrow\B(\H)$ in which $\phi(\chi_k)=T_k$, 
 for $k=1,\dots, d$. Inequality (b) asserts that $\psi$ is completely contractive; hence, the map
 $\phi:\mathcal Z_d+\mathcal Z_d^*\rightarrow \B(\H)$ given by $\phi(f+g^*)=\psi(f)+\psi(g)^*$ is unital and completely positive, and has the
 property that $\phi(\chi_k)=T_k$ and $\phi(\chi_k^*)=T_k^*$, for $k=1,\dots, d$. Because $\chi_k=\chi_1^k$, for $k=1,\dots, d$, the tuple 
 $(\chi_1,\dots,\chi_d)$ of unitaries is Toeplitz contractive; hence, 
 so is the $d$-tuple $(\phi(\chi_1),\dots,\phi(\chi_d))=(T_1,\dots, T_d)$,
 by the complete positivity of the linear map $\phi$.
 
To complete the proof, assume that $\H$ is a separable Hilbert space.
The C$^*$-algebra generated by the multiplication operator 
$M_\chi$ is $\{M_\varphi\,|\,\varphi\in C(S^1)\}$, which is a faithful $*$-representation
of $C(S^1)$, and $\cstar(M_\chi)$ contains
no compact operators other than $0$. 
Thus, for any unital completely positive linear map $\phi:\cstar(M_\chi)\rightarrow\B(\H)$, 
there exists a sequence of isometric operators $\tilde V_n:\H\rightarrow L^2(S^1)$ such that
\begin{equation}\label{e:v1}
\lim_{n\rightarrow\infty}\|\phi(M_\varphi)-\tilde V_n^*M_\varphi\tilde V_n\|=0,
\end{equation}
for every $\varphi\in C(S^1)$ \cite[Theorem II.5.3]{Davidson-book}, \cite{voiculescu1976}. 
In particular, assuming ${\sf T}$ is Toeplitz contractive,  the previous paragraph shows the existence of a unital completely positive linear map
$\phi:\cstar(M_\chi)\rightarrow\B(\H)$ such that $T_k=\phi(M_{\chi_k})$, for $k=1,\dots, d$. Hence,
if $\tilde U:\ell^2(\mathbb Z)\rightarrow L^2(S^1)$
is the unitary operator for which $W=\tilde U^*M_\chi \tilde U$, where $W$ is the bilateral shift, then
(\ref{e:v1}) becomes
\begin{equation}\label{e:v2}
\lim_{n\rightarrow\infty}\|T_k-V_n^*W^kV_n\|=0,
\end{equation}
for $k=1,\dots,d$, where each $V_n:\H\rightarrow\ell^2(\mathbb Z)$ is the isometry given by $V_n=\tilde U\tilde V_n$.
 \end{proof} 
 
Items (b) and (d) in Theorem \ref{main result} above are natural outcomes of similar ideas 
in \cite[Theorem 4.1]{farenick--floricel--plosker2013} and 
\cite[Theorem 5.2]{davisdon--dor-on--shalit--solel2017}, respectively. 
The assertion about Toeplitz contractive $d$-tuples acting on a finite-dimensional Hilbert space 
having a unitary power dilation that also acts on a space of finite dimension can be proved entirely algebraically,
as in \cite{gurvits2001}, or by means of convex analysis, as in \cite{farenick--mcburney2023}. The latter approach is not unlike the
recent work of Hartz and Lupini \cite{hartz--lupini2021}, where matrix convexity is shown to play a prominent role in the
dilation theory for operators acting on finite-dimensional Hilbert spaces.

It is well know that the unitary dilation $U$ of a contraction $T$ explicitly constructed by Halmos in \cite{halmos1950} 
does not have the property that $U^n$ dilates $T^n$, for every $n\in\mathbb N$. In this respect, assertion (c) of
Theorem \ref{main result} may be viewed as ``Halmos dilation theorem in several variables.''

\begin{corollary}\label{fd1} If $T$ is a contraction acting on a Hilbert space $\H$, then
there exists a Hilbert space $\K$ containing $\H$ as a subspace and a unitary operator $U\in\B(\K)$
such that
\begin{equation}\label{e:fd sz-nagy}
T^k=P_\H U^k_{\vert\H},
\end{equation}
for all $k=1,\dots, d$.

Furthermore, if $\H$ has finite dimension, then the dilating Hilbert space $\K$ can be chosen to have finite dimension.
\end{corollary}

\begin{proof} Example \ref{ex:power contraction} shows the $d$-tuple ${\sf T}=(T,T^2,\dots,T^d)$ is Toeplitz contractive. 
Thus, by Theorem \ref{main result}, ${\sf T}$ has a dilation to a power unitary
$d$-tuple ${\sf U}=(U,\dots, U^d)$, and the dilating space $\K$ has finite dimension, if $\H$ has finite dimension.
\end{proof}

Corollary \ref{fd1} is a truncated version of the 
famous dilation theorem of Sz.-Nagy \cite{Sz-Nagy-book}, which asserts every contraction $T$ admits a unitary dilation $U$
for which (\ref{e:fd sz-nagy}) holds for all $k\in\mathbb N$, not just for $k=1,\dots,d$. The dilating Hilbert space $\K$ in Sz.-Nagy's theorem
is generally infinite-dimensional, even if $\H$ has finite dimension---for example, this is the case if $T$ is a non-unitary contraction.
Therefore, the
interest in Corollary \ref{fd1} lies in the proof, independent of Sz.-Nagy's theorem and of the similar result of Egerv\'ary \cite{egervary1954}, 
of the existence of a unitary power dilation, up to $d$ powers,
and the assertion that $\K$ can be chosen to have finite dimension, if $\H$ has finite dimension.

A similar truncated version of Berger's ``strange dilation theorem'' \cite{berger1965} holds for numerical
contractions on finite-dimensional spaces. However, the only new
information imparted in Corollary \ref{fd2} below is 
the assertion concerning the finite-dimensionality of the dilating space, if the original space has
finite dimension. The existence of the dilation is not given by a new argument; in this regard, Corollary \ref{fd2}  is 
in the same vein as \cite[Theorem 7.1]{davisdon--dor-on--shalit--solel2017}, which shows that once a normal dilation is known to exist, a
normal dilation to a finite-dimensional space can be found.

\begin{corollary}\label{fd2} If $S$ is an operator acting on a finite-dimensional Hilbert space $\H$ such that $w(S)\leq 1$, then
there exists a finite-dimensional Hilbert space $\K$ containing $\H$ as a subspace and a unitary operator $U\in\B(\K)$
such that
\begin{equation}\label{e:fd berger}
S^k=2P_\H U^k_{\vert\H},
\end{equation}
for all $k=1,\dots, d$.
\end{corollary}

\section{The Toeplitz Modulus and Numerical Radius on $\B(\H)^d$}

\begin{definition}\label{d:Tm} The \emph{Toeplitz modulus} of ${\sf T}=(T_1,\dots,T_d)\in\B(\H)^d$ is the nonnegative real
number $\rho({\sf T})$ defined by
\begin{equation}\label{e:Tm}
\rho({\sf T})= \inf\left\{ \alpha\in\mathbb R_+\,|\,
\left[ \begin{array}{ccccc} \alpha 1_\H & T_1^* & T_2^* & \dots & T_d^* \\
                                                              T_1 & \alpha 1_\H & T_1^* & \ddots & \vdots \\
                                                              T_2 & \ddots & \ddots & \ddots &  T_2^* \\
                                                              \vdots & \ddots & \ddots & \ddots & T_1^* \\
                                                              T_d & \dots & T_2 & T_1 & \alpha 1_\H 
\end{array}\right] \mbox{ is positive}\right\}.
\end{equation}
\end{definition}

In light of Definition \ref{d:Tm}, an operator $d$-tuple ${\sf T}$ is Toeplitz contractive if and only if $\rho({\sf T)}\leq 1$.

Some basic properties of the Toeplitz modulus are presented in the next few results. To facilitate their
proofs, the following notation shall be used.

For any ${\sf A}=(A_1,\dots, A_d)\in\B(\H)^d$ and $\alpha\in\mathbb R_+$, 
denote $\mathcal A_\alpha$ by
\begin{equation}\label{e:Aalpha}
\mathcal A_\alpha =
\left[ \begin{array}{ccccc} \alpha 1_\H & A_1^* & A_2^* & \dots & A_d^* \\
                                                              A_1 & \alpha 1_\H & A_1^* & \ddots & \vdots \\
                                                              A_2 & \ddots & \ddots & \ddots &  A_2^* \\
                                                              \vdots & \ddots & \ddots & \ddots & A_1^* \\
                                                              A_d & \dots & A_2 & A_1 & \alpha 1_\H 
\end{array}\right].
\end{equation}

Thus, the defining condition (\ref{d:Tm}) for the Toeplitz-modulus of ${\sf T}$ becomes
\begin{equation}\label{d:Tm2}
\rho({\sf T})= \inf\left\{ \alpha\in\mathbb R_+\,|\,\mathcal T_\alpha\mbox{ is a positive operator}\right\}.
\end{equation}

\begin{proposition}\label{prop1}
Suppose ${\sf T}=(T_1,\dots,T_d)\in\B(\H)^d$ and define ${\sf T}^*=(T_1^*,\dots,T_d^*)$.
\begin{enumerate}
\item If $d=1$, then $\rho({\sf T})=\|T_1\|$.
\item $\rho({\sf T}^*)=\rho({\sf T})$.
\item If $1\leq q\leq d$ and $1\leq i_1< \cdots <i_q\leq d$, then
\[
\rho\left( (T_{i_1}, \dots, T_{i_q})\right) \leq \rho({\sf T}),
\]
and $\|T_k\|\leq \rho({\sf T})$, for every $k=1,\dots, d$.
\item If $\phi:\B(\H)\rightarrow\B(\K)$ is a ucp map, the $\rho\left(\phi({\sf T}) \right) \leq \rho({\sf T})$.
\item $\rho({\sf T})=0$ if and only if ${\sf T}=0$.
\end{enumerate}
\end{proposition}

\begin{proof} The first assertion was noted earlier in equation (\ref{e:n1}).

For (2), for each $\alpha\in\mathbb R_+$, the Toeplitz matrices of operators $\mathcal T_\alpha^*$ and $\mathcal T_\alpha$ are
unitarily equivalent (via the selfadjoint permutation of $\displaystyle\bigoplus_1^d\H$ that repositions the direct summands 
of $\displaystyle\bigoplus_1^d\H$ in reverse order).
Hence, $\mathcal T_\alpha^*$ is positive if and only if $\mathcal T_\alpha$ is positive.

To prove (3), assume $1\leq i_1<i_2<\cdots< i_q\leq d$ and let ${\sf T}(\mathcal I)=(T_{i_1},T_{i_2},\dots,T_{i_q})$.
Thus, for each $\alpha\in\mathbb R_+$, the operator matrix $\mathcal T(\mathcal I)_\alpha$ is a $q\times q$
principal submatrix of $\mathcal T_\alpha$. Thus, the positivity of $\mathcal T_\alpha$ implies the positivity of
$\mathcal T(\mathcal I)_\alpha$. Hence, for any $\alpha\in\mathbb R_+$ for which $\mathcal T_\alpha$ is positive, we have
$\rho\left( {\sf T}({\mathcal I})\right) \leq \alpha$, and so $\rho\left( {\sf T}({\mathcal I})\right) \leq \rho({\sf T})$.

The proof (4) is similar to the proof of (3): if $\alpha\in\mathbb R_+$ is such that $\mathcal T_\alpha$ is positive, then 
$\Phi(\mathcal T)_\alpha$ is also positive, where $\Phi(\mathcal T)_\alpha$ is the Toeplitz matrix of operators obtained by 
applying $\phi$ to each entry of $\mathcal T_\alpha$.

Select $k\in\{1,\dots,d\}$. Using (3) with $\mathcal I$ determined by the singleton set $\{k\}$, we have $\|T_k\|\leq \rho({\sf T})$.
Hence, $\rho({\sf T})=0$ if and only if ${\sf T}=0$, thereby proving (5).
\end{proof}

An important, interesting property of the Toeplitz modulus is that it is a convex function $\B(\H)^d\rightarrow\mathbb R$.

\begin{proposition}\label{Tm is convex}
If ${\sf A},{\sf S}, {\sf T}\in \B(\H)^d$ and $a,s,t\in\mathbb R_+$, then
\[
\rho(a{\sf A})=a\rho({\sf A}) \mbox{ and }
\rho(s{\sf S}+t{\sf T}) \leq s\rho({\sf S})+t\rho({\sf T}).
\]
\end{proposition}

\begin{proof}
The equation $\rho(a{\sf A})=a\rho({\sf A})$ holds trivially if $a=0$. If $a>0$, then 
\[
\left[ \begin{array}{ccccc} \alpha 1_\H & aA_1^* & aA_2^* & \dots &a A_d^* \\
                                                              aA_1 & \alpha 1_\H & aA_1^* & \ddots & \vdots \\
                                                              aA_2 & \ddots & \ddots & \ddots &  aA_2^* \\
                                                              \vdots & \ddots & \ddots & \ddots & aA_1^* \\
                                                             a A_d & \dots & aA_2 & aA_1 & \alpha 1_\H 
\end{array}\right]
=
a
\left[ \begin{array}{ccccc} \frac{\alpha}{a} 1_\H & A_1^* & A_2^* & \dots & A_d^* \\
                                                              A_1 & \frac{\alpha}{a} 1_\H & A_1^* & \ddots & \vdots \\
                                                              A_2 & \ddots & \ddots & \ddots &  A_2^* \\
                                                              \vdots & \ddots & \ddots & \ddots & A_1^* \\
                                                              A_d & \dots & A_2 & A_1 & \frac{\alpha}{a} 1_\H 
\end{array}\right]=a \mathcal A_{\frac{\alpha}{a}}.
\]
Because the matrix on the left is positive if and only if the matrix $\mathcal A_{\frac{\alpha}{a}}$ is positive, we have
\[
\rho(a {\sf A}) = a \cdot \inf\left\{\frac{\alpha}{a}\,| \mathcal A_{\frac{\alpha}{a}} \mbox{ is positive}\right\} = a\rho({\sf A}).
\]

To complete the proof, it is sufficient to show that $\rho$ is subadditive, as 
$\rho(a{\sf A})=a\rho({\sf A})$ for every $a\in\mathbb R_+$.
To this end, select ${\sf S}, {\sf T}\in \B(\H)^d$ and suppose $\beta,\gamma\in\mathbb R_+$ are such that
$\mathcal S_\beta$ and $\mathcal T_\gamma$ are positive operators. Thus, the 
operator $\mathcal S_\beta+\mathcal T_\gamma$ is also positive.
Let $\alpha=\beta+\gamma$. Then
by the defining condition (\ref{e:Tm}), 
\[
\rho({\sf S}+{\sf T}) \leq \alpha = \beta+\gamma.
\]
As the inequality above is true for any $\beta,\gamma\in\mathbb R_+$
for which $\mathcal S_\beta$ and $\mathcal T_\gamma$ are positive operators, then passing to the infima of all such $\beta$
and $\gamma$ leads to
\[
\rho({\sf S}+{\sf T}) \leq \rho({\sf S})+\rho({\sf T}),
\]
which shows that $\rho$ is subadditive.
\end{proof}

One drawback of the Toeplitz modulus is that it fails to be a norm, or even a metric, if $d>1$, as the following example 
shows. 

\begin{example}[Asymmetry]\label{ex:nsym} If $\H=\mathbb C^2$ and $\sigma=\left[ \begin{array}{cc} 0&1 \\ 1&0 \end{array}\right]$, then the pair
$(\sigma, 1_2)$ is Toeplitz contractive, but $(-\sigma,-1_2)$ is not. Moreover, 
\[
\rho\left((\sigma, 1_2)\right)=1<\rho\left(-(\sigma, 1_2)\right).
\]
\end{example}

\begin{proof} Because $\sigma$ is a selfadjoint unitary, 
the pair $(\sigma, 1_2)$ has the form $(U,U^2)$, which is Toeplitz-contractive. Hence,
\[
\left[ \begin{array}{ccc} 1_2 & \sigma & 1_2 \\ \sigma & 1_2 & \sigma \\ 1_2 & \sigma & 1_2 \end{array}\right]
\]
is a positive operator. 
On the other hand, the block-Toeplitz matrix
\[
\left[ \begin{array}{ccc} 1_2 & -\sigma & -1_2 \\ -\sigma & 1_2 & -\sigma \\ -1_2 & -\sigma & 1_2 \end{array}\right]
\]
has an eigenvalue of $-1$. Thus, $(-\sigma,-1_2)$ is not a Toeplitz contractive pair, implying its Toeplitz-modulus
exceeds $1$.

To show $\rho\left( (\sigma,1_2)\right)=1$, write ${\sf T}=(\sigma,1_2)$ and
suppose $\alpha\in\mathbb R_+$ is such that the matrix $\mathcal T_\alpha$ is positive. 
Let $\xi\in\mathbb C^2$ be a unit (eigen)vector
satisfying $\sigma\xi=\xi$, and let $\phi$ denotes the vector state $X\mapsto\langle X\xi,\xi\rangle$. Applying the ampliation
$\phi^{[3]}$ to the positive matrix $\mathcal T_\alpha$ yields the positive matrix
\[
\left[ \begin{array}{ccc} \alpha & 1 &  1 \\ 1 & \alpha & 1 \\ 1 & 1 & \alpha\end{array}\right]. 
\]
However, by considering the upper-left $2\times 2$ submatrix, the matrix above is positive only if $\alpha\geq1$.
Hence,  $1\leq\rho\left( (\sigma,1_2)\right)\leq1$.
\end{proof}

Thus, if one were to use the Toeplitz-modulus as a distance measure, it may happen that the distance from ${\sf S}$ to ${\sf T}$
is different from the distance from ${\sf T}$ to ${\sf S}$. A remedy to this lack of symmetry in the set of Toeplitz-contractive
operators is to take an average.

\begin{proposition} The formula $D^{\rho}({\sf S},{\sf T})=\frac{1}{2}\left(\rho({\sf S}-{\sf T})+\rho({\sf T}-{\sf S}) \right)$
defines a metric on $\B(\H)^d$.
\end{proposition}

\begin{proof} The only non-obvious property of a metric to verify is the triangle inequality. If
${\sf S}, {\sl T}, {\sf R}\in\B(\H)^d$, then
\[
\begin{array}{rcl} 
D^\rho({\sf S}, {\sf T}) &=& \frac{1}{2}\rho({\sf S}- {\sl T}) + \frac{1}{2}\rho({\sf T}- {\sf S}) \\ &&\\
&=& \frac{1}{2}\left( \rho([{\sf S}- {\sf R}]+[{\sf R}-{\sf T}]) + \rho([{\sf T}- {\sf R}]+[{\sf R}-{\sf S}]) \right) \\ && \\
&\leq& \frac{1}{2}\left( \rho([{\sf S}- {\sf R}])+\rho([{\sf R}-{\sf T}]) + \rho([{\sf T}- {\sf R}])+\rho([{\sf R}-{\sf S}]) \right) \\ && \\
&=&  \frac{1}{2}\left( \rho([{\sf S}- {\sf R}])+\rho([{\sf R}-{\sf S}]\right) +  \frac{1}{2}\left( \rho([{\sf T}- {\sf R}])+\rho([{\sf R}-{\sf T}])\right) \\ && \\
&=& D^\rho({\sf S}, {\sf R}) + D^\rho({\sf R}, {\sf T}),
\end{array}
\]
where the inequality is due to the subadditivity of the Toeplitz-modulus.
Hence, $D^\rho$ is a metric.
\end{proof} 

Because the interest of the present paper is with dilations, further study of $\B(\H)^d$ as a $D^\rho$-metric space is not 
undertaken here. Again,
from the point of view of dilations, the metric $D^\rho$ is not entirely satisfying, as
the closed unit ball in the $D^\rho$-metric excludes some Toeplitz contractive $d$-tuples,
such as the pair $(\sigma,1_2)$ in Example \ref{ex:nsym}.

The following definition is a new generalistion of the numerical radius of an operator $d$-tuple.

\begin{definition}\label{d:Tm nr} If ${\sf T}=(T_1,\dots,T_d)\in\B(\H)^d$ and $(\alpha,\xi)\in\mathbb R_+\times\H$, and if
$\mathcal T_{\alpha,\xi}$ denotes the Toeplitz matrix
\begin{equation}\label{e:T nr matrix}
\mathcal T_{\alpha,\xi}=
\left[ \begin{array}{ccccc} \alpha & \langle T_1^*\xi,\xi\rangle & \langle T_2^*\xi,\xi\rangle & \dots & \langle T_d^*\xi,\xi\rangle \\
                                                             \langle T_1 \xi,\xi\rangle  & \alpha   & \langle T_1^*\xi,\xi\rangle & \ddots & \vdots \\
                                                            \langle T_2 \xi,\xi\rangle & \ddots & \ddots & \ddots &  \langle T_2^*\xi,\xi\rangle \\
                                                              \vdots & \ddots & \ddots & \ddots & \langle T_1^*\xi,\xi\rangle \\
                                                              \langle T_d \xi,\xi\rangle & \dots & \langle T_2 \xi,\xi\rangle & \langle T_1 \xi,\xi\rangle & \alpha 
\end{array}\right],
\end{equation}
then the
\emph{Toeplitz numerical radius} of ${\sf T}$ is the nonnegative real
number $\omega({\sf T})$ defined by
\begin{equation}\label{e:T nr}
\omega({\sf T}) =\inf\left\{ \alpha\in\mathbb R_+\,|\,\mathcal T_{\alpha,\xi} \mbox{ is positive for every unit vector }\xi\in\H\right\}.
\end{equation}
\end{definition}

Observe that the first column of each matrix in (\ref{e:T nr matrix}) is determined by $\alpha$
and an element of the spatial numerical range of ${\sf T}$.

As with the Toeplitz modulus, the Toeplitz numerical radius is a convex 
function, but is not homogenous for nonpositive or nonreal scalars. However, if $d=1$, then 
Definition \ref{d:Tm nr} yields $\omega({\sf T})=w(T_1)$, 
the classical numerical radius of the operator $T_1$.

The Toeplitz numerical radius is defined spatially. However, as with the classical joint numerical radius, there is a useful
alternative description in terms of states on $\B(\H)$.

\begin{proposition}\label{nr alg} If $\phi:\B(\H)\rightarrow\mathbb C$ is a state, ${\sf T}=(T_1,\dots,T_d)\in\B(\H)^d$, 
and $\alpha \in\mathbb R_+ $, and if
$\mathcal T_{\alpha,\phi}$ denotes the Toeplitz matrix
\begin{equation}\label{e:T nr matrix state}
\mathcal T_{\alpha,\phi}=
\left[ \begin{array}{ccccc} \alpha & \phi( T_1^*) & \phi(T_2^*) & \dots & \phi(T_d^*) \\
                                                             \phi( T_1 )  & \alpha   & \phi(T_1^*) & \ddots & \vdots \\
                                                            \phi( T_2 ) & \ddots & \ddots & \ddots &  \phi (T_2^*) \\
                                                              \vdots & \ddots & \ddots & \ddots & \phi( T_1^*) \\
                                                              \phi( T_d ) & \dots & \phi(T_2) & \phi (T_1 ) & \alpha 
\end{array}\right],
\end{equation}
then 
\begin{equation}\label{e:T nr alg}
\omega({\sf T}) =\inf\left\{ \alpha\in\mathbb R_+\,|\,\mathcal T_{\alpha,\phi}\; \mbox{\rm  is positive for every state }\phi\;
\mbox{\rm  on }\B(\H) \right\}.
\end{equation}
\end{proposition}

\begin{proof}
As every unit vector $\xi\in\H$ induces a state $X\mapsto\langle X\xi,\xi\rangle$
on $\B(\H)$, $\omega({\sf T})$ is no larger than the quantity on the
right side of equation (\ref{e:T nr alg}). 

Conversely, assume $\alpha\in\mathbb R_+$ has the property that $\mathcal T_{\alpha,\xi}$ is positive, 
for every unit vector $\xi$. As every state on $\B(\H)$ restricts to a state 
on the unital C$^*$-algebra
$\cstar({\sf T})$ generated by $T_1,\dots,T_d$, select a pure state $\phi$ on $\cstar({\sf T})$.
Because $\phi$ is pure and $\cstar({\sf T})$ is separable, there
is a sequence $\{\xi_n\}_{n\in\mathbb N}$ of unit vectors $\xi_n\in\H$ such that 
$\phi(X)=\displaystyle\lim_{n\rightarrow\infty}\langle X\xi_n,\xi_n\rangle$, for every $X\in\cstar({\sf T})$
\cite[Lemma 2.1]{bunce--salinas1976}. Hence, the matrices $\mathcal T_{\alpha,\xi_n}$ converge entrywise to 
$\mathcal T_{\alpha,\phi}$ as $n\rightarrow\infty$, which implies the matrices
$\mathcal T_{\alpha,\xi_n}$ converge in norm to
$\mathcal T_{\alpha,\phi}$. By hypothesis, $\mathcal T_{\alpha,\xi_n}$ is positive for every $n$;
hence, $\mathcal T_{\alpha,\phi}$ is positive.

Because $\mathcal T_{\alpha,\phi}$ is positive for every pure state $\phi$ on $\cstar({\sf T})$, the same is true for any
state $\phi$ on $\cstar({\sf T})$ that is a convex combination of pure states.

Assume $\phi$ is an arbitrary state on $\cstar({\sf T})$. By the Krein-Milman Theorem, there is a net $\{\phi_\mu\}_\mu$ of
convex combinations of pure states on $\cstar({\sf T})$ such that $\phi(X)=\displaystyle\lim_\mu \phi_\mu(X)$, for every
$X\in\cstar({\sf T})$. Hence, the net $\{\mathcal T_{\alpha,\phi_\mu}\}_\mu$ of positive matrices converges pointwise, and hence in norm,
to $\mathcal T_{\alpha,\phi}$, which proves that $\mathcal T_{\alpha,\phi}$ is positive. 

Hence, the infimum of $\left\{ \beta\in\mathbb R_+\,|\,\mathcal T_{\beta,\phi}\; \mbox{\rm  is positive for every state }\phi\;
\mbox{\rm  on }\B(\H) \right\}$ is bounded above by $\alpha$,
which proves 
the equality (\ref{e:T nr alg}).
\end{proof}

Although the Toeplitz numerical radius is defined spatially, one advantage in using the definition involving states on $\B(\H)$ is that
the set of matrices of the form $\mathcal T_{\alpha,\phi}$ that are positive, for a fixed $\alpha$ and arbitrary state $\phi$, is convex.
In contrast, the spatially defined joint numerical range of a $d$-tuple of operators can fail to be convex, implying that the set of matrices
of the form $\mathcal T_{\alpha,\xi}$ that are positive, for a fixed $\alpha$ and arbitrary unit vector $\xi$, can fail to be convex.

An inequality \cite[\S18]{Halmos-book-SM} arising from the polarisation identity 
and the characterisation of the operator norm $\|T\|$
in terms of the bilinear form $(\xi,\eta)\mapsto\langle T\xi,\eta\rangle$ on $\H\times \H$
leads to the following well-known inequalities involving the operator norm and numerical radius:
\[
\frac{1}{2}\|T\|\leq w(T) \leq \|T\|.
\]

The analogous upper bound for the Toeplitz numerical radius and modulus
is straightforward to confirm, as shown by Proposition \ref{ub} below.

\begin{proposition}\label{ub} $ \omega({\sf T})\leq \rho({\sf T})$.
\end{proposition}

\begin{proof}
Suppose that $\alpha\in\mathbb R_+$ is such that $\mathcal T_\alpha$ is a positive operator.
For any unit vector $\xi\in\H$, the positive linear functional $\phi_\xi(X)=\langle X\xi,\xi\rangle$, for $X\in\B(\H)$, is a unital completely positive
linear map. Therefore, the matrix $\phi_\xi^{[d+1]}(\mathcal T_\alpha)$ is positive in $\M_{d+1}(\mathbb C)$, where $\phi_\xi^{[d+1]}$ denotes
the $(d+1)$-th ampliation of $\phi_\xi$. Because the matrix $\phi_\xi^{[d+1]}(\mathcal T_\alpha)$ is the matrix given in (\ref{e:T nr}), $\alpha$
is an upper bound for $\omega({\sf T)}$. As this is true for any $\alpha\in\mathbb R_+$ for which $\mathcal T_\alpha$ is a positive operator, 
$\rho({\sf T})$ is also an upper bound for $\omega({\sf T)}$.
\end{proof}

As with the classical numerical radius, the Toeplitz numerical radius coincides with the Toeplitz modulus for $d$-tuples of commuting normal
operators. To prove this, it is convenient to introduce a version of spectral radius.

\begin{definition}\label{d:Tm sr} If ${\sf T}=(T_1,\dots,T_d)\in\B(\H)^d$ is a $d$-tuple of commuting operators
and $(\alpha,\lambda)\in\mathbb R_+\times\spec_T({\sf T})$, and if
$\Lambda_{\alpha,\lambda}$ denotes the Toeplitz matrix
\begin{equation}\label{e:T sr matrix}
\Lambda_{\alpha,\lambda}=
\left[ \begin{array}{ccccc} \alpha & \overline{\lambda}_1 & \overline{\lambda}_2 & \dots & \overline{\lambda}_d \\
                                                            \lambda_1    & \alpha   & \overline{\lambda}_1 & \ddots & \vdots \\
                                                            \lambda_2   & \ddots & \ddots & \ddots &  \overline{\lambda}_2 \\
                                                              \vdots & \ddots & \ddots & \ddots & \overline{\lambda}_1 \\
                                                              \lambda_d   & \dots & \lambda_2 & \lambda_1 & \alpha 
\end{array}\right],
\end{equation}
then the
\emph{Toeplitz spectral radius} of ${\sf T}$ is the nonnegative real
number $r({\sf T})$ defined by
\begin{equation}\label{e:T sr}
r({\sf T}) =\inf\left\{ \alpha\in\mathbb R_+\,|\,\Lambda_{\alpha,\lambda} \mbox{ is positive for every }\lambda\in \spec_T({\sf T})\right\}.
\end{equation}
Further, if $T$ is a normal $d$-tuple, then the \emph{Gelfand-Toeplitz spectral radius} of ${\sf T}$ 
is the nonnegative real
number $r_G({\sf T})$ defined by
\begin{equation}\label{e:T sr}
r_G({\sf T}) =\inf\left\{ \alpha\in\mathbb R_+\,|\,\Lambda_{\alpha,\lambda} \mbox{ is positive for every }\lambda\in \spec_G({\sf T})\right\}.
\end{equation}

\end{definition}

\begin{proposition}\label{max Tm normal} If ${\sf T}$ is a $d$-tuple of commuting operators, then
\[
r({\sf T})\leq \omega({\sf T})\leq \rho({\sf T}).
\]
If ${\sf N}$ is a normal $d$-tuple, then
\[
r_G({\sf N})=r({\sf N})=\omega({\sf N})=\rho({\sf N}).
\]
\end{proposition}

\begin{proof} By \cite[Corollary 2.3]{wrobel1988}, each $\lambda\in\spec_T({\sf T})$ lies in the closure of the 
spatial numerical range of ${\sf T}$. Thus, if $\alpha\in\mathbb R_+$ is such that the matrix $\mathcal T_{\alpha,\xi}$ is positive
for every unit vector $\xi$, then the matrix $\Lambda_{\alpha,\lambda}$ is positive for every $\lambda\in\spec_T({\sf T})$.
Hence, $r({\sf T})\leq\alpha$, which implies $r({\sf T})\leq \omega({\sf T})$.

Suppose now that ${\sf N}$ is normal. If $\pi:\cstar({\sf N})\rightarrow C\left( \spec_G({\sf N})\right)$
is the unital $*$-isomorphism representing the unital abelian C$^*$-algebra
$\cstar({\sf N})$ as the C$^*$-algebra $C\left( \spec_G({\sf N})\right)$ of
continuous complex-valued functions on its maximal ideal space
$\spec_G({\sf N})$, then each $\pi(N_k)$ is the function $\lambda\mapsto \lambda_k$. Furthermore, 
the ampliation 
\[
\pi^{[d+1]}:\M_{d+1}\left(\cstar({\sf N})\right)\rightarrow \M_{d+1}\left( C(\spec_G({\sf N}))\right),
\]
in which $[X_{ij}]_{i,j}\mapsto [\pi(X_{ij}]_{i,j}$, is also a unital $*$-ismorphism.

Assume $\alpha\in\mathbb R_+$ is such that
$\Lambda_{\alpha,\lambda}$ is positive for every $\lambda\in\spec_G({\sf N})$. Thus, the matrix-valued function 
$F_\alpha:\spec_G({\sf N})\rightarrow\M_{d+1}(\mathbb C)$ given by $F_\alpha(\lambda)=\Lambda_{\alpha, \lambda}$ is a positive element
in the C$^*$-algebra $\M_{d+1}\left( C(\spec_G({\sf N}))\right)$. Since $F_\alpha=\pi^{[d+1]}(\mathcal T_{\alpha})$ and because
$\pi^{[d+1]}$ is a $*$-isomorphism, $\mathcal T_\alpha$ is positive. Hence, $\rho({\sf N})\leq \alpha$, which implies 
$\rho({\sf N})\leq r_G({\sf N})$. On the other hand, 
$r_G({\sf N})\leq r({\sf N})$ by definition, and so $r_G({\sf N})\leq \rho({\sf N})$.
\end{proof}

\begin{corollary} If ${\sf S}$ is a subnormal $d$-tuple, then $r({\sf S})=\omega({\sf S})=\rho({\sf S})$.
\end{corollary}

\begin{proof} If ${\sf N}$ is a minimal normal extension for ${\sf S}$, then 
\[
\spec_T({\sf N})\subseteq \spec_T({\sf S})\subseteq \overline{W({\sf S})} \subseteq \overline{W({\sf N})}, 
\]
where the first two inclusions are established in \cite{putinar1984} and \cite[Corollary 2.3]{wrobel1988}, respectively. Because
Proposition \ref{max Tm normal} shows the Gelfand-Toeplitz spectral radius of 
${\sf N} $ coincides with the Toeplitz numerical radius and modulus of
${\sf N}$, we deduce $r({\sf S})=\omega({\sf S})=\rho({\sf S})$.
\end{proof}

\section{Embedding $S^1$ into $\mathbb C^d$ and Matrix Convexity}

\begin{definition}
Given a Hilbert space $\H$, $d\in\mathbb N$, and $r\in\mathbb R_+$, let 
\[
\mbox{\rm Toepl}_{d,\H}^\rho(r)=\left\{{\sf T}\in\B(\H)^d\,|\,\rho({\sf T})\leq r\right\}
\]
and 
\[
\mbox{\rm Toepl}_{d,\H}^\omega(r)=\left\{{\sf T}\in\B(\H)^d\,|\,\omega({\sf T})\leq r\right\}.
\]
If $\H=\mathbb C^n$ for some $n\in\mathbb N$, then $\mbox{\rm Toepl}_{d,\H}^\rho(r)$ 
and $\mbox{\rm Toepl}_{d,\H}^\omega(r)$
are denoted by $\mbox{\rm Toepl}_{d,n}^\rho(r)$ and $\mbox{\rm Toepl}_{d,n}^\omega(r)$.
\end{definition}

If $\K$ is any Hilbert space, and $X:\K\rightarrow \H$ is a bounded linear operator, then $X$ induces a linear
map $\B(\H)^d\rightarrow\B(\K)^d$, denoted by ${\sf T}\mapsto X^*\cdot {\sf T} \cdot X$, and defined by
\[
X^*\cdot{\sf T}\cdot X= X^*\cdot(T_1,\dots,T_d)\cdot X =(X^*T_1X, \dots, X^*T_dX).
\]

\begin{definition}
An \emph{operator convex combination} of ${\sf T}_1,\dots,{\sf T}_\ell\in \B(\H)^d$ is a sum of the form
\[
\sum_{k=1}^\ell X_k^* \cdot {\sf T}_k \cdot X_k,
\]
where $X_1,\dots, X_\ell$ are any bounded linear operators $\K\rightarrow\H$, for any Hilbert space $\K$, 
for which
\[
\sum_{k=1}^\ell X_k^*X_k=1_\K.
\]
A subset of $\B(\H)^d$ which is closed under all possible operator convex combinations is said to be \emph{operator convex}.
\end{definition} 

\begin{proposition}\label{op-cnv}
For every $d\in\mathbb N$, Hilbert space $\H$, and $r\in\mathbb R_+$, the sets
$\mbox{\rm Toepl}_{d,\H}^\rho(r)$ 
and
$\mbox{\rm Toepl}_{d,\H}^\omega(r)$ are operator convex.
\end{proposition}

\begin{proof} 
Because $\rho(r{\sf T})=r\rho({\sf T})$, we may assume without loss of generality that $r=1$.
Suppose ${\sf T}_1,\dots,{\sf T}_\ell\in \B(\H)^d$ are Toeplitz contractive, $\K$ is a Hilbert space, and
$X_1,\dots, X_\ell:\K\rightarrow\H$ are operators satisfying
$\displaystyle\sum_{k=1}^\ell X_k^*X_k=1_\K$. In writing
\[
{\sf T}_k=(T_{k1},\dots,T_{kd}) \mbox{ and }
{\sf S} =(S_1, \dots, S_d) = \sum_{k=1}^d X_k^* \cdot {\sf T}_k \cdot X_k,
\]
we obtain
\begin{equation}\label{e:forS}
S_j=\sum_{k=1}^\ell X_k^*T_{kj} X_k, \mbox{ for each }j=1,\dots, d.
\end{equation}
Thus,
\[
\left[ \begin{array}{ccccc}   1_\H & S_1^* & S_2^* & \dots & S_d^* \\
                                                              S_1 &   1_\H & S_1^* & \ddots & \vdots \\
                                                              S_2 & \ddots & \ddots & \ddots &  S_2^* \\
                                                              \vdots & \ddots & \ddots & \ddots & S_1^* \\
                                                              S_d & \dots & S_2 & S_1 &   1_\H 
\end{array}\right] 
=
\sum_{k=1}^\ell \tilde X_k^* \left[ \begin{array}{ccccc}   1_\H & T_{k1}^* & T_{k2}^* & \dots & T_{kd}^* \\
                                                              T_{k1} &   1_\H & T_{k1}^* & \ddots & \vdots \\
                                                              T_{k2} & \ddots & \ddots & \ddots &  T_{k2}^* \\
                                                              \vdots & \ddots & \ddots & \ddots & T_{k1}^* \\
                                                              T_{kd} & \dots & T_{k2}& T_{k1} &   1_\H 
\end{array}\right] \tilde X_k,
\]
where $\tilde X_k = \mbox{Diag}(X_k, X_k, \dots X_k)$. As the right side of the
equation above is a sum of positive operators, ${\sf S}$ is a Toeplitz
contractive $d$-tuple.

Similarly, to prove $\mbox{\rm Toepl}_{d,\H}^\omega(1)$ is operator convex for $r=1$,
select any unit vector $\eta\in\K$ and consider, using the assumptions and notation established above, the matrix
$\mathcal S_{1,\eta}$. Let $\hat\eta_k$ be $0$, if $X_k\eta=0$, or $\|X_k\eta\|^{-1}(X_k\eta)$ otherwise.
By equation (\ref{e:forS}), 
\[
\langle S_j\eta,\eta\rangle = \sum_{k=1}^\ell \langle T_{kj}X_k\eta, X_k\eta\rangle
=
\sum_{k=1}^\ell \|X_k\eta\|^2\langle T_{kj} \hat\eta_k, \hat\eta\rangle,
\]
for each $j=1,\dots,d$, and so
\[
\mathcal S_{1,\eta} = \sum_{k=1}^\ell \|X_k\eta\|^2 \mathcal T_{1,\hat\eta_k}.
\]
Since ${\sf T}\in \mbox{\rm Toepl}_{d,\H}^\omega(1)$, the matrices $\mathcal T_{1,\hat\eta_k}$ are positive, implying the
matrix $\mathcal S_{1,\eta}$ is positive.
\end{proof}

Proposition \ref{op-cnv} shows that the (graded) set of all Toeplitz contractive operators is a noncommutative convex set 
in the sense of \cite{kennedy--kim--manor2023}; likewise, for the set of all $d$-tuples of operators 
with Toeplitz numerical radius $1$ (or any $r>0$).
However, for the remainder of this section, we shall require only the notion of a matrix convex set.

\begin{definition}\label{d:m-cnvx}
A \emph{matrix convex set in $d$ free dimensions} is a graded set $\mathbb K$ of the form 
$\mathbb K=\displaystyle\bigcup_{n\in\mathbb N} \mathbb K_n$,
where 
\begin{enumerate}
\item $\mathbb K_n\subseteq\M_n(\mathbb C)^d$, and
\item for all $\ell,m,n_1,\dots,n_\ell\in\mathbb N$, $A_j\in\mathbb K_{n_j}$, and linear maps
$X_j:\mathbb C^m\rightarrow\mathbb C^{n_j}$ such that $\displaystyle\sum_{j=1}^\ell X_j^*X_j=1_m$, 
we have
\[
\sum_{j=1}^\ell X_j^*\cdot A_j \cdot X_j\in\mathbb K_n.
\]
\end{enumerate}
\end{definition}

Define $\mbox{\rm Toepl}_{d}^\rho(r)$ and $\mbox{\rm Toepl}_{d}^\omega(r)$ to be the graded sets
\[
\mbox{\rm Toepl}_{d}^\rho(r)=\displaystyle\bigcup_{n\in\mathbb N}\mbox{\rm Toepl}_{d,n}^\rho(r)\,\mbox{ and }\,
\mbox{\rm Toepl}_{d}^\omega(r)=\displaystyle\bigcup_{n\in\mathbb N}\mbox{\rm Toepl}_{d,n}^\omega(r),
\]
respectively. 
Proposition \ref{op-cnv} shows that each of $\mbox{\rm Toepl}_{d}^\rho(r)$ and $\mbox{\rm Toepl}_{d}^\omega(r)$ are matrix convex. 
Because the cone of positive operators is norm-closed, these matrix convex sets are also compact, which is to say each of the sets at the $n$-th grading
is compact in $\M_n(\mathbb C)^d$.
Another class of compact matrix convex sets arise from the matrix ranges (Definition \ref{d:qaz}, defining condition (\ref{e:jmra})) 
of operator $d$-tuples ${\sf T}\in\B(\H)$.

\begin{definition} The \emph{matrix range} of ${\sf T}\in\B(\H)^d$ is the graded set
$\mathbb W({\sf T})$ defined by
\[
\mathbb W({\sf T})= \bigcup_{n\in\mathbb N} W_n({\sf T}).
\]
\end{definition}

The matrix convexity of $\mathbb W({\sf T})$ is evident from the fact that linear maps of the form 
$A\mapsto \displaystyle\sum_{j=1}^\ell X_j^*  A  X_j$ are completely positive, and are unital when 
$\displaystyle\sum_{j=1}^\ell X_j^*  X_j$ is the identity matrix.

\begin{theorem}\label{min-pre}  The following
statements are equivalent for ${\sf A}=(A_1,\dots,A_d)\in\M_n(\mathbb C)^d$;
\begin{enumerate}
\item ${\sf A}\in \mathbb W({\sf W})$, the unitary-power tuple determined by the bilateral shift operator $W$;
\item ${\sf A}$ is Toeplitz contractive;
\item there exist $\ell\in\mathbb N$,  
$\lambda_1,\dots,\lambda_\ell\in S^1$, and $Q_1,\dots,Q_\ell\in\M_n(\mathbb C)_+$ such that
\[
\sum_{j=1}^\ell Q_j=1_n\mbox{ and } A_k=\sum_{j=1}^\ell \lambda_j^k Q_j, \mbox{ for all }k=1,\dots,d.
\]
\end{enumerate}
\end{theorem}

\begin{proof}
If ${\sf A}\in \mathbb W({\sf W})$, then there is a ucp map $\phi:\B(\ell^2(\mathbb Z))\rightarrow\M_n(\mathbb C)$ such that
$\phi(U^k)=A_k$, for all $k=1,\dots, d$. Thus, the matrix $\mathcal A_{1}$ is given by $\phi^{[d+1]}(\mathcal U_1)$. The positivity
of $\mathcal U_1$ implies, therefore, the positivity of  $\mathcal A_{1}$, which is the assertion that the $d$-tuple ${\sf A}$ is
Toeplitz contractive.

Next, assume ${\sf A}$ is
Toeplitz contractive. As the matrix $\mathcal A_1$ 
is a $(d+1)\times(d+1)$ matrix with entries that are $n\times n$ matrices $A_k$, $\mathcal A_1$ can be viewed as a positive
element of $\M_{d+1}\otimes\M_n(\mathbb C)$. By the Gurvits separability theorem \cite[Corollary 3.3]{farenick--mcburney2023}, 
the positive block-Toeplitz matrix
$\mathcal A_1$ is separable and, moreover, each
entry of $\mathcal A_1$ admits a decomposition of the form 
$A_k=\displaystyle\sum_{j=1}^\ell \lambda_j^k Q_j$, where 
$\lambda_1,\dots,\lambda_\ell\in S^1$ and $Q_1,\dots,Q_\ell\in\M_n(\mathbb C)_+$ are such that
$\displaystyle\sum_{j=1}^\ell Q_j=1_n$.

Finally, assume $A_1,\dots,A_d\in\B(\H)$ admit decompositions as indicated in (3): $\displaystyle\sum_{j=1}^\ell Q_j=1_n$ and
$A_k=\displaystyle\sum_{j=1}^\ell \lambda_j^k Q_j$, for some $\lambda_1,\dots,\lambda_\ell\in S_1$.
Because $\mbox{\rm Sp}(W)=S^1$, there exist characters $\pi_j:\cstar({\sf W})\rightarrow\mathbb C$, for $j=1,\dots,\ell$, such that
$\pi_j(W^k)=\lambda_j^k$, for $k=1,\dots,d$.
Therefore, the linear map $\phi:\cstar({\sf W})\rightarrow\M_n(\mathbb C)$ defined by
$\phi(X)=\displaystyle\sum_{j=1}^\ell \pi_j(X)Q_j$, for $X\in\cstar({\sf W})$, is unital and completely positive, and has the property that
${\sf A}=\phi({\sf W})$. Hence, ${\sf A}\in \mathbb W({\sf W})$.
\end{proof}

The theory of minimal and maximal matrix convex sets \cite{passer--shalit--solel2018}
concerns those matrix convex sets $\mathbb K$ for which $\mathbb K_1$ is some prescribed 
compact convex set $\Omega\subseteq\mathbb C^d$. 
As demonstrated in \cite{passer--shalit--solel2018}, among all
possibilities, 
there are a smallest and a largest such matrix convex set, and 
these are denoted by $\mathbb K^{\rm min}(\Omega)$ and $\mathbb K^{\rm max}(\Omega)$, respectively.  
These matrix convex sets have rather precise descriptions:
a $d$-tuple ${\sf A}\in \mathbb K^{\rm min}(\Omega)$ if and only if
there exists a normal dilation ${\sf N}$ of ${\sf A}$ such that $\mbox{\rm Sp}_G({\sf N})\subseteq\Omega$, whereas 
a $d$-tuple ${\sf B}\in \mathbb K^{\rm max}(\Omega)$ if and only if the joint numerical range of ${\sf B}$ is a subset of $\Omega$.

\begin{definition}\label{d:stretched}
A \emph{power stretching of $S^1$ across $d$ complex dimensions} is the 
set $\mathbb S_d^1\subseteq\mathbb C^d$ defined by
\[
\mathbb S_d^1=\left\{(\lambda,\lambda^2,\dots,\lambda^d)\in\mathbb C^d\,|\,\lambda\in S^1\right\}.
\]
\end{definition}

Observe that $\mathbb S_d^1$ is a subgroup of the $d$-torus, 
isomorphic to the group $S^1$, which is why $\mathbb S_d^1$ is referred to as a 
``stretching'' of $S^1$.
For our purposes, 
the importance of $\mathbb S_d^1$ lies in the fact that it coincides with Gelfand joint spectrum of the unitary-power $d$-tuple
${\sf W}=(W,W^2,\dots, W^d)$, where $W$ is the bilateral shift operator on $\ell^2(\mathbb Z)$. 

\begin{theorem}\label{minmax} If ${\sf A}, {\sf B}\in\M_n(\mathbb C)^d$, then
\begin{enumerate}
\item ${\sf A}\in \mathbb K^{\rm min}(\mbox{\rm conv}\,\mathbb S_d^1)$ if and only if $\rho({\sf A})\leq 1$, and
\item ${\sf B}\in \mathbb K^{\rm max}(\mbox{\rm conv}\,\mathbb S_d^1)$ if and only if $\omega({\sf B})\leq 1$.
\end{enumerate}
\end{theorem}

\begin{proof} To prove (1), assume assume ${\sf A}\in \mathbb K^{\rm min}(\mbox{\rm Conv}\,\mathbb S_d^1)$. Then, by definition of $\mathbb K^{\rm min}(\cdot)$,
${\sf A}$ admits a unitary-power dilation ${\sf U}$ such that the joint spectrum of ${\sf U}$ is a subset of $\mathbb S_d^1$.
Thus, $\rho({\sf A})\leq\rho({\sf U})=1$.

Conversely, if ${\sf A}$ is a $d$-tuple of matrices such that $\rho({\sf A})\leq 1$, then, by Theorem \ref{main result},
${\sf A}$ admits a unitary-power dilation ${\sf U}$ such that the joint spectrum of ${\sf U}$ is a subset of $\mathbb S_d^1$. Hence, 
${\sf A}\in \mathbb K^{\rm min}(\mbox{\rm Conv}\,\mathbb S_d^1)$.

To prove (2), assume ${\sf B}\in \mathbb K^{\rm max}(\mbox{\rm Conv}\,\mathbb S_d^1)$, select a unit vector $\xi$, and consider the
matrix $\mathcal B_{1,\xi}\in\M_{d+1}(\mathbb C)$. By hypothesis, $W^1({\sf B})\in \mbox{\rm Conv}\,\mathbb S_d^1$; thus,
there exist $\ell\in\mathbb N$, $t_1,\dots,\t_\ell\in\mathbb R_+$, and $\lambda_1,\dots,\lambda_\ell\in\mathbb S^1$ such that
$\displaystyle\sum_{j=1}^\ell t_j=1$ and $\langle B_k,\xi,\xi\rangle=\displaystyle\sum_{j=1}^\ell t_j\lambda_j^k$, for each $k=1,\dots,d$.
Therefore, $\mathcal B_{1,\xi}=\displaystyle\sum_{j=1}^\ell t_j T_{d+1}(\lambda_j)$, which is a convex combination of positive matrices
$T_{d+1}(\lambda_j)$ (as defined in Definition \ref{d:uTm}). 
Hence, $\mathcal B_{1,\xi}$ is positive. Since the choice of $\xi$ is arbitrary, this establishes $\omega({\sf B})\leq 1$.

Conversely, if ${\sf B}$ is a $d$-tuple of matrices such that $\omega({\sf B})\leq 1$, then, by Theorem \ref{Toepl ext pts},
for any unit vector $\xi$, the positive matrix
$\mathcal B_{1,\xi}$ is a convex combination $\mathcal B_{1,\xi}=\displaystyle\sum_{j=1}^\ell t_j T_{d+1}(\lambda_j)$ of matrices of the
form $T_{d+1}(\lambda_j)$, for some $\lambda_1,\dots,\lambda_\ell\in\mathbb S^1$. Hence, 
$\langle B_k,\xi,\xi\rangle=\displaystyle\sum_{j=1}^\ell t_j\lambda_j^k$, for each $k=1,\dots,d$, which implies $W^1({\sf B})\subseteq \mbox{\rm Conv}\,\mathbb S_d^1$.
That is, ${\sf B}\in \mathbb K^{\rm max}(\mbox{\rm Conv}\,\mathbb S_d^1)$.
\end{proof}

Because convexity theory is generally undertaken in the setting of real vectors, we shall need to pass to real and imaginary parts for both spaces and operators. 
In considering $\mathbb C^d$ as $\mathbb R^d+i\mathbb R^d\cong\mathbb R^{2d}$, $\mathbb C^d$ is a real vector space of real dimension $2d$,
and convex subsets $\Omega\subseteq\mathbb C^d$ are also convex when considered as subsets of $\mathbb R^{2d}$. Importantly, if 
$\mathbb B_{\ell^2(n)}^{\mathbb C}$ and $\mathbb B_{\ell^2(n)}^{\mathbb R}$ denote the closed unit balls of $\mathbb C^n$ and $\mathbb R^n$, respectively, with respect to the
Euclidean $\ell^2$ norm, then 
\[
\mathbb B_{\ell^2(d)}^{\mathbb C}=\mathbb B_{\ell^2(2d)}^{\mathbb R},
\]
in considering $\mathbb B_{\ell^2(d)}^{\mathbb C}$ as a subset of $\mathbb R^{2d}$.

\begin{proposition}\label{contain} $\frac{1}{\sqrt{d}}\mathbb B_{\ell^2(d)}^{\mathbb C} \subseteq \mbox{\rm Conv}\,\mathbb S_d^1 
\subseteq \sqrt{d}\,\mathbb B_{\ell^2(d)}^{\mathbb C}$.
\end{proposition}

\begin{proof} The inclusion $\mbox{\rm Conv}\,\mathbb S_d^1 \subseteq \sqrt{d}\,\mathbb B_{\ell^2(d)}^{\mathbb C}$ follows from the fact that
every element of $\mathbb S_d^1$ has norm $\sqrt{d}$.

Conversely, select a vector $\xi\in\mathbb C^d$ of norm $\|\xi\|\leq\frac{1}{\sqrt{d}}$, and consider the Toeplitz matrix $\Lambda_{1,\xi}$.
To show that $\xi\in \mbox{\rm Conv}\,\mathbb S_d^1 $, it is sufficient to prove that $\Lambda_{1,\xi}$ is positive. Thus, by the Ger\v sgorin Circle Theorem, it is
sufficient to show the sum of the moduli of the off-diagonal entries in each row (or column) of $\Lambda_{1,\xi}$ is no larger than $1$. 
Because of the Toeplitz structure of $\Lambda_{1,\xi}$,
it sufficient to show this for the first column. By the Cauchy-Schwarz inequality,
\[
\sum_{j=1}^{d} |\xi_j| \leq \sqrt{d} \left(\sum_{j=1}^{d} |\xi_j|^2 \right)^{1/2} = \sqrt{d}\, \|\xi\| \leq 1.
\]
Thus, $\Lambda_{1,\xi}$ is positive, and so $\xi\in \mbox{\rm Conv}\,\mathbb S_d^1$.
\end{proof}

\begin{corollary}\label{cnvx body} The convex hull of $\mathbb S_d^1$ is a convex body in $\mathbb C^d$
containing $0$ in its interior.
\end{corollary}

The importance of Corollary \label{cnvx body} lies in the fact that the results of \cite[\S3.2]{passer--shalit--solel2018} can be applied to 
$\mbox{\rm Conv}\,\mathbb S_d^1$. To this end, the following Banach-Mazur-type distance, applicable to nonsymmetric convex bodies, was 
used to great value in \cite{passer--shalit--solel2018}. Namely, if $K$ and $L$ are
two compact convex subsets of $\mathbb R^{2d}$
such that $K \cap L\supseteq\{0\}$, then
\[
\delta(K,L)=\inf\left\{C>0\,|\,\mbox{there exists }R\in\mbox{\rm GL}_{2d}(\mathbb R)\mbox{ such that } K\subseteq R(L) \subseteq C\cdot K\right\}
\]
is finite. 

In the case where $K=\mathbb B_{\ell^2(d)}^{\mathbb C}$ and $L=\mbox{\rm Conv}\,\mathbb S_d^1$, Proposition \ref{contain} shows that
\[
\delta\left(\mathbb B_{\ell^2(d)}^{\mathbb C},\mbox{\rm Conv}\mathbb S_d^1\right) \leq d,
\]
where in the set for which $\delta$ is the infimum, one choice of $R$ and $C$ is given by $R=\sqrt{d}\cdot 1_{2d}$
and $C=\sqrt{d}\,\sqrt{d}=d$.

\begin{theorem}\label{cnst}{\rm (\cite[\S5]{passer--shalit--solel2018})} If $\Omega$ is a compact convex body in $\mathbb C^d$
with $0$ in its interior, then there exists a positive constant $\theta(\Omega)$
such that
\begin{enumerate}
\item $\mathbb K^{\rm max}(\Omega) \subseteq \theta(\Omega) \cdot \mathbb K^{\rm min}(\Omega)$,
\item $\theta(\Omega)=\min\left\{C>0\,|\,\mathbb K^{\rm max}(\Omega) \subseteq C \cdot \mathbb K^{\rm min}(\Omega)\right\}$, and
\item $2d \leq \delta\left(\mathbb B_{\ell^2(d)}^{\mathbb C},\Omega\right)\,\theta(\Omega)$.
\end{enumerate}
\end{theorem}

The final main result of this paper is an analogue, for the Toeplitz modulus and numerical radius, of the
familar inequality $\|T\|\leq 2w(T)$ in operator theory. However, the exact value of the constant $c_d$ is not known for $d\geq 2$.

\begin{theorem}\label{lb} For every $d\in\mathbb N$, there exists a constant $c_d\geq 2$ such that
\[
\rho({\sf T})\leq c_d\cdot\omega({\sf T}),
\]
for every operator $d$-tuple $T\in\B(\H)^d$.
\end{theorem}

\begin{proof} The set $\mbox{\rm Conv}\,\mathbb S_d^1$ is a compact, convex body with $0$ in its interior, by 
Proposition \ref{contain}. Therefore, Theorem \ref{cnst} applies to $\mbox{\rm Conv}\,\mathbb S_d^1$, leading to the 
existence of a positive constant $c_d=\theta(\mbox{\rm Conv}\,\mathbb S_d^1)$ with the properties
\[
\mathbb K^{\rm max}\left(\mbox{\rm Conv}\,\mathbb S_d^1\right) \subseteq 
c_d
\cdot 
\mathbb K^{\rm min}\left(\mbox{\rm Conv}\,\mathbb S_d^1\right)
\]
and
\[
2d \leq \delta\left(\mathbb B_{\ell^2(d)}^{\mathbb C}\,\mbox{\rm Conv}\,\mathbb S_d^1\right)\,c_d.
\]
Because $\delta\left(\mathbb B_{\ell^2(d)}^{\mathbb C}\,\mbox{\rm Conv}\,\mathbb S_d^1\right)\leq d$, we obtain from the inequality above that
$c_d\geq 2$.

Suppose now that ${\sf T}\in\B(\H)$ and let $\alpha=\omega({\sf T})$. Thus, $\tilde{\sf T}=\frac{1}{\alpha}{\sf T}$ has Toeplitz numerical radius
$\omega(\tilde{\sf T})=1$. By Theorem \ref{minmax}, $\tilde{\sf T}\in \mathbb K^{\rm max}\left(\mbox{\rm Conv}\,\mathbb S_d^1\right)$. Therefore, 
\[
\tilde{\sf T}\in \cdot 
\mathbb K^{\rm min}\left(\mbox{\rm Conv}\,\mathbb S_d^1\right),
\]
which implies $\rho(c_d^{-1}\tilde{\sf T})\leq 1 = \omega(\tilde{T})$ and, hence, $\rho({\sf T})\leq c_d\cdot\omega({\sf T})$.
\end{proof}

\section{Discussion}

The material presented here is inspired by the dilation theorem of Halmos and the theorems of
Ando and Gurvits on truncated moments and separable positive block Toeplitz matrices. 
As explained earlier, the underlying geometry for Halmos' dilation theorem is the geometry of
$\mbox{\rm Conv}\,S^1$, whereas the geometry inherent to the Ando-Gurvits theorems relies on the stretched version $\mathbb S_d^1$ of $S^1$
and its convex hull. Some earlier works, such as \cite{binding--farenick--li1995,fritz--netzer--thom2017}, established a dilation theory
for operator $d$-tuples in which the normal dilations have joint spectrum within a simplex. Works such as 
\cite{davisdon--dor-on--shalit--solel2017,passer--shalit--solel2018} advanced the theory of normal dilations substantially, showing that 
simplex-structures could be replaced by more general regions. 
Likewise, in the present paper, the 
geometry of $\mbox{\rm Conv}\,\mathbb S_d^1$ is not that of a simplex, but it is a subset of $\mathbb C^d$
in which the joint spectrum of unitary power dilations of Toeplitz-contractive $d$-tuples can be found. 

It is of interest to determine good estimates, or the exact value, for the constant $c_d$ in Theorem \ref{lb}. The lack of symmetry
in $\mbox{\rm Conv}\,\mathbb S_d^1$ is an obstacle in drawing upon known estimates for the Banach-Mazur distance between
$\mathbb B_{\ell^2(d)}^\mathbb C$ and other (usually symmetric) convex bodies.

\section*{Acknowledgement}

I am indebted to Tim Netzer for comments in support of the idea that a result such as the one presented in
Theorem \ref{lb} could be true, and to Igor Klep, who noted a small but crucial typo in my original version of
Example \ref{pairs}.

This work has been supported, in part, by a Discovery Grant awarded
by the Natural Sciences and Engineering Research Council of Canada.

\bibliographystyle{amsplain}

\end{document}